\documentclass[a4paper,reqno]{amsart}

\usepackage{a4wide, enumerate, amsmath, amsfonts, amssymb, amsthm, wasysym, graphics, graphicx, xcolor, url, hyperref, hypcap, xargs, stackrel, footnote, multirow}

\usepackage{todonotes}

\usepackage{epsfig}
\usepackage{a4wide}
\usepackage{enumerate}
\usepackage{tikz}
\usepackage{verbatim}
\usepackage{color}

\newtheorem{theorem}{Theorem}[section]
\newtheorem{corollary}[theorem]{Corollary}
\newtheorem{defi}{Definition}
\newtheorem{example}{Example}
\newtheorem{lemma}[theorem]{Lemma}
\newtheorem{proposition}[theorem]{Proposition}

\newcommand{\C}{\mathcal{C}}
\newcommand{\M}{\mathcal{M}}
\newcommand{\N}{\mathcal{N}}
\newcommand{\U}{\mathcal{U}}
\newcommand{\B}{\mathcal{B}}
\newcommand{\rk}{\operatorname{rank}}

\newcommand{\y}{\boldsymbol{y}}
\newcommand{\0}{\boldsymbol{0}}

\newcommand{\NN}{\mathcal{N}}

\newcommand{\ring}{\mathsf{R}}
\newcommand{\multiedge}{\mathsf{M}}
\newcommand{\unif}{\mathsf{U}}
\newcommand{\UMR}{\mathsf{UMR}}

\newcommand{\AR}{A_{\ring}}
\newcommand{\AM}{A_{\multiedge}}
\newcommand{\AU}{A_{\unif}}

\newcommand{\aU}{a_{\unif}}

\newcommand{\solAR}{\widehat{\AR}}
\newcommand{\solAU}{\widehat{\AU}}

\newcommand{\TU}{T_{\unif}}
\newcommand{\tU}{t_{\unif}}

\newcommand{\SR}{S_{\ring}}
\newcommand{\SM}{S_{\multiedge}}
\newcommand{\SU}{S_{\unif}}

\newcommand{\NR}{N_{\ring}}
\newcommand{\NM}{N_{\multiedge}}
\newcommand{\NU}{N_{\unif}}

\DeclareMathOperator{\mul}{\operatorname{Mul}}

\newcommand{\etalchar}[1]{$^{#1}$}
\newcommand\Defn[1]{\emph{\color{black}#1}}
\newcommand\1{\mathbf{1}}
\newcommand\Char[1]{\1_{#1}}

\DeclareMathOperator{\conv}{\mathrm{conv}}

\begin{document}

\title[Many 2-level polytopes from matroids]{Many 2-level polytopes from matroids}

\author{Francesco Grande}
\address{(FG) Freie Universit\"at Berlin, Institut f\"ur Mathematik und Informatik, Arnimallee 2, 14195 Berlin, Germany}
\email{fgrande@zedat.fu-berlin.de}
\urladdr{http://page.mi.fu-berlin.de/grande/}

\author{Juanjo Ru\'e}
\address{(JR) Freie Universit\"at Berlin, Institut f\"ur Mathematik und Informatik, Arnimallee 3, 14195 Berlin, Germany}
\email{jrue@zedat.fu-berlin.de}
\urladdr{http://www-ma2.upc.edu/jrue/}

\thanks{F.\,G.~was supported by the DFG within the research training group
\emph{Methods for Discrete Structures} (GRK1408).\\
\indent
J.\,R.~was partially supported by the Spanish MICINN grant MTM2011-22851, the FP7-PEOPLE-2013-CIG project CountGraph (ref. 630749), the DFG within the research training group \emph{Methods for Discrete Structures} (GRK1408), and the \emph{Berlin Mathematical School}.
}
\maketitle




\begin{abstract}
The family of $2$-level matroids, that is, matroids whose base polytope is $2$-level, has been recently studied and characterized by means of combinatorial properties. $2$-level matroids generalize series-parallel graphs, which have been already successfully analyzed from the enumerative perspective.

We bring to light some structural properties of $2$-level matroids and exploit them for enumerative purposes. Moreover, the counting results are used to show that the number of combinatorially non-equivalent $(n {-} 1)$-dimensional 2-level polytopes is bounded from below by $c \cdot n^{-5/2} \cdot \rho^{-n}$, where $c\approx 0.03791727 $ and $\rho^{-1} \approx 4.88052854$.

\end{abstract}




\section{Introduction}
A hyperplane $H$ is \Defn{facet-defining} for a polytope $P$ if it is supporting for $P$ and $\dim(P\cap H)=\dim(P)-1$. A \Defn{$2$-level polytope} is a polytope $P$ such that for each facet-defining hyperplane $H$, there exists a hyperplane $H'$ parallel to $H$ that contains all the vertices of $P$ not in $H$. The family of $2$-level polytopes appeared in the literature in different areas under different names: in \cite{Santos} they are called compressed polytopes and also show up in statistics (see \cite{Sullivant}).
In the context of combinatorial optimization \cite{Parrilo} and \cite{Laurent}, $2$-level polytopes are related to the so-called exact point configurations: the interest in these configurations is due to the fact that some techniques from polynomial optimization, namely semidefinite programming relaxations, are very efficient for these configurations. Furthermore	$2$-level polytopes play a role in the study of extremal centrally-symmetric polytopes \cite{SWZ}.

Two polytopes are \Defn{combinatorially equivalent} if their face lattices are isomorphic. It is known that all $2$-level $n$-dimensional polytopes are affinely equivalent to $0/1$-polytopes (polytopes with all vertices in $\{  0,1 \}^n$) and the number of combinatorially non-equivalent $0/1$-polytopes is doubly-exponential in the dimension (see \cite{Lecton01pol}). Among the finite number of $0/1$-polytopes of fixed dimension, it is natural to ask how many are $2$-level, up to combinatorial equivalence.

Though $2$-level polytopes are endowed with a very restrictive geometric property, this class is not well-understood and an exact enumeration seems to be complicated. It is easy to see that the $2$-levelness is preserved for some polytopal constructions: pyramid, prism, and Cartesian product.
Moreover some subfamilies of $2$-level polytopes are known: two of them are explored in \cite{Schmitt}, the so-called \Defn{Hansen polytopes} \cite{Hansen} and \Defn{Hanner polytopes} \cite{Hanner}, while a third one arises from stable sets of perfect graphs as explained in \cite[Ch.~9]{Lovasz}. Note that the construction of twisted prism over this last family yields the family of Hansen polytopes. \Defn{Order polytopes} of finite posets \cite{Stanleyorderpol} are also $2$-level. Very recently, a new subfamily of $2$-level polytopes arising from matroid theory has been characterized in \cite{Grande}. More precisely, this subfamily is associated with the base polytopes of \Defn{$2$-level matroids}.

A complete classification of the $0/1$-equivalence classes of $0/1$-polytopes is only available for dimension $3$, $4$, $5$, (and $6$ for polytopes up to $12$ vertices). Moreover two polytopes that are $0/1$-equivalent are also combinatorially equivalent, but the converse is not true. The difficulties in providing a complete list already in dimension $6$ suggest that a computational approach to the problem could be unsuccessful. The lack of an exact enumeration in dimension $\geq 6$ leads to a second natural question, namely the existence of asymptotic bounds for the number of $2$-level polytopes.

By means of the polytopal constructions we mentioned above (pyramid, prism, and Cartesian product) exponentially many combinatorially non-equivalent $2$-level polytopes can be constructed. In this paper we compute an explicit exponential lower bound for the number of $2$-level polytopes via $2$-level matroids. More precisely, we prove the following theorem.
\begin{theorem}\label{thm:main}
The number of combinatorially non-equivalent $(n{-}1)$-dimensional 2-level polytopes is bounded from below by $$c \cdot n^{-5/2} \cdot \rho^{-n},$$
where $c\approx 0.03791727 $ and $\rho^{-1}$ is a computable constant whose value is approximately equal to $4.88052854$.
\end{theorem}
The interest in the subfamily of $2$-level matroids is motivated by the fact that it contains more complicated polytopes, namely not obtained by means of elementary polytopal constructions. Moreover it allows to determine a large basis for the exponential lower bound.

New combinatorial aspects of $2$-level matroids are introduced in Section \ref{sect: matroids-decom} and give the possibility to increase the control on the enumerative formulas. It is noteworthy that this matroid family generalizes the family of \Defn{series-parallel graphs}, which appears in various areas and has several interesting properties that are likely to have counterparts for 2-level matroids. In particular, series-parallel graphs have been already successfully studied from an enumerative point of view in \cite{SP} and \cite{DrmotaFusyKangKrausRue}. To approach the enumeration of $2$-level matroids we investigate the \Defn{matroid tree decomposition} associated to these matroids.
We analyze the features of the decomposition and we get one of the main results of the paper: we observe that there is an interesting interpretation in terms of acyclic structures. More precisely, we reveal a bijection between $2$-level matroids and a family of trees, that we call \Defn{$\UMR$-trees}, whose vertices are labelled by uniform matroids and satisfy some adjacency restrictions. This last discovery makes $2$-level matroids particularly suitable for enumeration. Indeed the family of $\UMR$-trees is exploited in Section~\ref{sect:counting} to encode all the enumerative information in terms of generating functions and relations (equations) among them by means of the symbolic method in enumerative combinatorics. Finally, powerful analytic techniques are applied to the equations in order to get an asymptotic estimate for the coefficients of the generating functions.
\\

\paragraph{\textbf{Structure of the paper}} The paper is structured in the following way: in Section~\ref{sec:prelim} the basics on matroid theory and enumerative combinatorics are stated. In Section~ \ref{sect: matroids-decom} we study how to decompose $2$-level matroids in terms of tree-like structures ($\UMR$-trees). The structural properties of the $\UMR$-trees are exploited later in Section~\ref{sect:counting} in order to get counting formulas which can be analyzed by means of analytic techniques, producing the estimate stated in Theorem~\ref{thm:main}.

\section{Preliminaries}\label{sec:prelim}

In this section we introduce the basic notions needed in the rest of the paper. In Subsection \ref{subsec:matroid} we focus on definitions and concepts related to matroid theory. In Subsection \ref{sub:prelim_enumeration} we fix our notation concerning enumeration by means of generating functions and finally in Subsection \ref{subsec: prelim-asympt} we state the results needed in order to get asymptotic estimates for the coefficients of the generating functions under consideration.

\subsection{Matroids}\label{subsec:matroid}

The basic definition is the following:

\begin{defi}\label{dfn:matroid}
A \Defn{matroid} of rank $k$ is an ordered pair $\M = (E,\B)$ consisting of a finite set $E$ (ground set) and a collection of bases $\emptyset\neq \mathcal{B}\subseteq  \binom{E}{k}$ satisfying the Basis Exchange Axiom: for $B_1,B_2 \in \B$ and $x \in B_1 \setminus B_2$ there exists $y\in B_2\setminus B_1$ such that $(B_1 \setminus x)\cup y \in \B$.
\end{defi}
Matroids are combinatorial objects that generalize graphs and linear dependence: the family of \Defn{graphic matroids} is particularly interesting and useful to visualize examples of matroids.
The matroid associated to a graph $G=(V,E)$ is such that the ground set is given by the set of edges and the collection of bases is given by the set of spanning forests. The rank of a connected graph is clearly $|V|-1$.
However, it could sometimes be misleading to think in terms of the graph structure, since some information, like the vertex structure, is not retained at matroid level.

A matroid has many equivalent definitions (see \cite{Oxley} for more details): we presented the one using the collection of bases. Nevertheless we want to introduce two further collections of sets that can define a matroid. The first one is the collection of \Defn{independent sets}, that is all the sets $X \subseteq E$ such that $X \subseteq B$, for some $B\in \B$. The rank of $X\subseteq E$, denoted by $\rk_{\M}(X)$, is the cardinality of the largest independent subset contained in $X$. The second one is the collection of \Defn{circuits} $\mathcal{C}(\M)$. Circuits are minimal dependent sets of $\M$. An element $e$ such that $\{e\}$ is a circuit is called a \Defn{loop}.

Two matroids $\M$ and $\N$ are \Defn{isomorphic} if their collection of circuits are the same up to relabelling of the ground sets $E(\M)$ and $E(\N)$. More formally $\M \cong \N$ if there is a bijection $\varphi  :  E(\M) \rightarrow E(\N)$ such that, for all $X\subseteq E(\M)$, $\varphi(X)\in \mathcal{C}(\N)$ if and only if $X\in \mathcal{C}(\M)$.

Let us consider a fairly simple family of matroids that is of great importance in the rest of the paper, namely uniform matroids. The \Defn{uniform matroid} $U_{n,k}$ for $0\leq k\leq n$ consists of the ground set $[n]:=\{1,\ldots , n  \}$ and the collection of bases $\binom{[n]}{k}$. The uniform matroids which are also graphic matroids are of the form: $U_{n,0}$, $U_{n,1}$, $U_{n,n{-}1}$, and $U_{n,n}$. See Figure~\ref{fig:matroids}.

\begin{figure}[htb]
\begin{center}
\includegraphics[width=10.5cm]{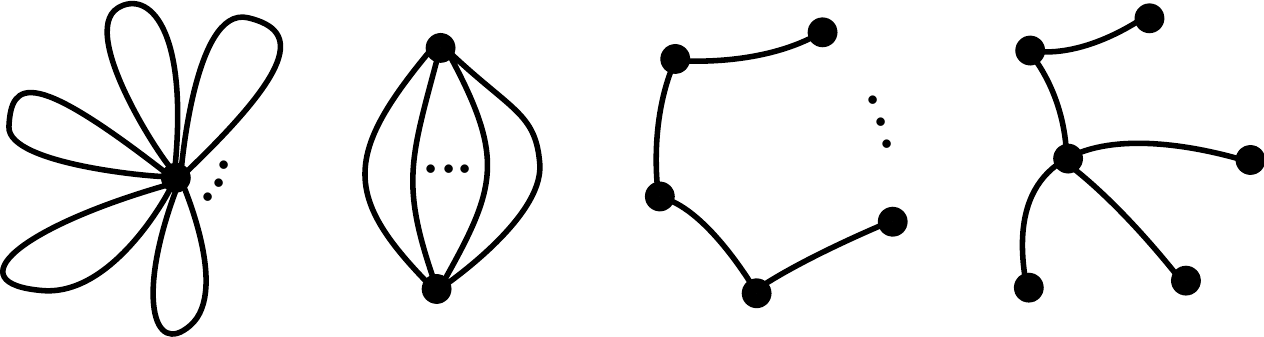}
\caption{From left to right, graphical representations of the matroids $U_{n,0}$, $U_{n,1}$, $U_{n,n{-}1}$, and $U_{n,n}$.}
\label{fig:matroids}
\end{center}
\end{figure}
Observe that for $U_{n,0}$ and $U_{n,n}$ we illustrated one among many possible graphical representations. Namely, Whitney's $2$-Isomorphism Theorem \cite[Thm.~5.3.1]{Oxley} implies that every graph formed by $n$ loops corresponds to $U_{n,0}$ regardless of the vertex structure, while any tree with $n$ edges corresponds to the matroid $U_{n,n}$.

For counting purposes we do not consider the uniform matroids $U_{n,0}$ and $U_{n,n}$, while among the other uniform matroids we need to distinguish the graphic ones from the non-graphic ones. More precisely we write $\multiedge_n$ to denote the matroid $U_{n,1}$ (it stands for \Defn{multiedge}) and $\ring_n$ to denote the matroid $U_{n,n{-}1}$ (it stands for \Defn{ring}).

The \Defn{dual matroid} $\M^*$ of a matroid $\M=(E,\B)$ is the matroid
defined by the pair $(E,\B^*)$ where $\B^*=\{E\setminus B \; : \; B\in \B
\} $. For uniform matroids we have $U_{n,k}^*=U_{n,n-k}$ and in particular $\ring_n^*=\multiedge_n$. An element $e$ is called a \Defn{coloop} of $\M$ if it is a loop of $\M^*$. A matroid $\M$ is \Defn{self-dual} if $\M\cong \M^*$. For instance, all uniform matroids of type $U_{2n,n}$ are self-dual.

\begin{defi}\label{dfn:base_conf}
Let $\M = (E,\B)$ be a matroid. The \Defn{base polytope} of $\M$ is the polytope
\[
    P_{\M} := \conv(\{ \Char{B} : B \in \B \} ).
    \]
\end{defi}
It was proven in \cite{GGMS} that all the edges of a base polytope are parallel to some difference $e_i-e_j$ of two unit vectors.

The base polytopes $P_{\ring_{n}}$ and $P_{\multiedge_{n}}$ are $n$-simplices, while the polytopes $P_{U_{n,k}}$ for $2\leq k \leq n{-}2$, are called \Defn{hypersimplices} and denoted by $\Delta_{n,k}$. For more background about this family of polytopes we refer to \cite{Zieglerhypersimp}.

A \Defn{$2$-level matroid} is a matroid such that the corresponding base polytope is $2$-level. In \cite{Grande} an excluded minor characterization for the family of $2$-level matroids is provided. The four excluded minors are the following rank $3$ matroids on $6$ elements: $\M(K_4)$, $\mathcal{W}^3$, $Q_6$, $P_6$. The first excluded minor of the list is nothing but the graphic matroid of the complete graph on $4$ vertices; for more details about these matroids we refer to Oxley's book \cite{Oxley} or to the paper where they are used to describe the $2$-level matroids \cite{Grande}. Since $P_6=([6],\B)$ appears in Example \ref{ex:P6}, we list here its collection of circuits
\[
\mathcal{C}(P_6)=\{123, 1245,1246, 1256, 1345,1346,1356, 1456,2345,2346,2356,2456,3456     \}.
\]
It is important to notice that there is only one circuit with $3$ elements. In \cite{Grande}, together with the excluded minor characterization of $2$-level matroids, a synthetic description of this class is also provided. Before presenting it, we need to introduce two matroid operations. Let $\M_1$ and $\M_2$ be matroids with disjoint ground sets $E_1$ and
$E_2$. The collection
\[
    \B \ := \ \{ B_1\cup B_2 :  B_1\in \B(\M_1), B_2\in \B(\M_2)\}
\]
is the set of bases of a matroid on $E_1 \cup E_2$, called the
\Defn{direct sum} of $\M_1$ and $\M_2$ and denoted by $\M_1 \oplus \M_2$. On the other hand, if we choose $e_1 \in E_1$ and $e_2\in E_2$ such that they are neither a loop nor a coloop of the respective matroids and define the collection
\[
    \B \ := \ \{ B_1\cup B_2\setminus \{e_1,e_2\}  :  B_1\in \B(\M_1), B_2\in
    \B(\M_2), |(B_1 \cup B_2) \cap \{e_1,e_2\}|=1 \},
\]
then the pair $(E_1 \cup E_2 \setminus \{ e_1, e_2 \} ,\B)$ defines a matroid called
\Defn{$2$-sum} of $\M_1$ and $\M_2$ with \Defn{base points} $e_1$ and $e_2$. We denote it by $(\M_1,e_1) \oplus_2 (\M_2,e_2)$.
Observe that this notation is slightly different from the one used in \cite{Oxley}, but it turns out to be more efficient for the constructive part.

\begin{theorem}[\cite{Grande}]\label{thm:unifdec2lev}
Every $2$-level matroid can be obtained as a sequence of direct sums and $2$-sums of uniform matroids. Moreover every combination of uniform matroids yields a $2$-level matroid.
\end{theorem}

The direct sum and the $2$-sum of matroids are closely related to the connectedness of a matroid: a matroid $\M$ is \Defn{$2$-connected} (or also \Defn{connected}) if it cannot be written as a proper direct sum of two matroids, and $\M$ is \Defn{3-connected} if it cannot be written as $2$-sum of two matroids each with fewer elements than $\M$.

A \Defn{separator} of a matroid $\M$ is a set $T\subseteq E$ such that $\rk_{\M}(T)+\rk_{\M}(E\setminus T)=\rk_{\M}(\M)$. A matroid $\M$ is $2$-connected if and only if there is no separator $T$, with $T$ being a proper subset of $E$.
The base polytope $P_{\M}$ of a matroid $\M=(E,\B)$ has dimension $|E|-c(\M)$ where $c(\M)$ is the number of $2$-connected components of $\M$. In particular, if $\M$ is $2$-connected, then $\dim (P_{\M})=|E|{-}1$.

If we try to look at matroid operations from the point of view of base polytopes we have:
\begin{itemize}
\item $P_{\M^*}=\1- P_{\M}$. This means that the base polytope of the dual matroid $P_{\M^*}$ is congruent to the base polytope $P_{\M}$;
\item $P_{\M_1\oplus \M_2}=P_{\M_1}\times P_{\M_2}$, where $\times$ denotes the Cartesian product of polytopes;
\item $P_{(\M_1,e_1)\oplus_2 (\M_2,e_2)}$ can be described using the subdirect product construction introduced in \cite{mcmullen} as shown in \cite{Grande}.
\end{itemize}

To keep the counting as easy as possible we first deal with $2$-connected matroids. This corresponds to considering only sequences of $2$-sums of uniform matroids. As a consequence, the polytopes we count cannot be obtained as a Cartesian product of two polytopes (for example no prism is in this family). At the end of Section \ref{sect:counting} we show that, asymptotically, the restriction to $2$-connected matroids does not alter the exponential growth.

The \Defn{basis graph} of a matroid $\M$ is the undirected graph with vertex set the collection of all bases of $\M$ such that a basis $B_1$ is connected to another basis $B_2$ whenever the symmetric difference $B_1 \Delta B_2$ has cardinality exactly $2$. Equivalently, it is the $1$-skeleton of the base polytope $P_\M$.

Let us conclude this section with some results for base polytopes that are used in Section \ref{sect:counting} to complete the asymptotic enumeration of $2$-level matroids. The first one appears as part of Exercise 4.9 in \cite[Ch.~4]{White}.

\begin{proposition}\label{prop:basisgraphs}
Let $\M$ and $\N$ be $2$-connected matroids. The basis graphs of $\M$ and $\N$ are isomorphic if and only if $\M \cong \N$ or $\M \cong \N^*$.
\end{proposition}

Since two congruent polytopes have the same $1$-skeleton we easily obtain the following corollary, which also appears as an exercise in \cite[Ch.~1, Ex.~18]{Coxeter}.

\begin{corollary}\label{Ardila}
Let $\M$ and $\N$ be $2$-connected matroids. The base polytopes $P_{\M}$ and $P_{\N}$ are congruent if and only if $\M\cong \N$ or $\M\cong \N^*$.
\end{corollary}

It is known that ``congruent'' $\Rightarrow$ ``combinatorially equivalent''. The converse is in general not true, not even for $0/1$-polytopes: for instance we can find full-dimensional $0/1$-simplices with different volumes \cite{Lecton01pol}. Nevertheless, for the class of base polytopes, we get the following corollary of Proposition \ref{prop:basisgraphs}.
\begin{corollary}
Let $\M$ and $\N$ be $2$-connected matroids. The polytope $P_{\M}$ is congruent to $P_{\N}$ if and only if $P_{\M}$ is combinatorially equivalent to $P_{\N}$.
\end{corollary}
\begin{proof}
We only need to prove one direction. If $P_{\M}$ is combinatorially equivalent to $P_{\N}$, then they have isomorphic face lattices and, in particular, isomorphic $1$-skeletons. By Proposition \ref{prop:basisgraphs}, $\M\cong \N$ or $\M\cong \N^*$ and therefore $P_{\M}$ is congruent to $P_{\N}$ by Corollary \ref{Ardila}.
\end{proof}
This last corollary allows us to investigate the number of non-congruent $2$-level base polytopes, instead of looking at combinatorial equivalence of such polytopes.

\subsection{The symbolic method in enumerative combinatorics. Tree-like structures} \label{sub:prelim_enumeration}

The reader is referred to ~\cite[Ch.~1]{FlajoletSedgewick} to see all the terminology and notation in full detail. Let $(\mathcal{A},|\cdot|)$ be an admissible combinatorial class, namely a set $\mathcal{A}$ endowed with a size function $|\cdot|$ such that the number of elements in $\mathcal{A}$ of any given size is finite. Then the \emph{generating function} (GF for short) associated to $\mathcal{A}$ is the formal power series $A(x)=\sum_{a\in \mathcal{A}} x^{|a|}=\sum_{n\geq 0}a_n x^n$. In particular, $a_n$ is the number of elements in $\mathcal{A}$ of size $n$ and we write $[x^n]A(x)=a_n$. We assume that every combinatorial class contains no object of size $0$, thus $a_0 = 0$. Given two generating functions $A(x)$ and $B(x)$, we write $A(x) \leq B(x)$ if for each $n$, $[x^n]A(x) \leq [x^n] B(x)$.

The \emph{symbolic method} in enumerative combinatorics (see~\cite[Ch.~1]{FlajoletSedgewick}) gives a direct way to translate combinatorial operations among combinatorial classes into operations involving their generating functions. Besides the disjoint union and Cartesian product of combinatorial families, which translate into sums and products of GFs, respectively, we introduce the multiset construction: given a combinatorial class $(\mathcal{A},|\cdot|)$ with GF $A(x)$, the \emph{multiset} of $\mathcal{A}$ is the combinatorial family obtained by taking all multisets of elements in $\mathcal{A}$. The corresponding GF is equal to
$$\mul(A(x))=\exp\left(\sum_{r=1}^\infty \frac{1}{r}A(x^r)\right).$$
Finally, we also need restricted multiset constructions. Let $\Lambda$ be a subset of positive integers. The multiset operator \emph{restricted to} $\Lambda$ of $\mathcal{A}$ is the combinatorial family obtained by taking multisets of elements in $\mathcal{A}$ with the restriction that the number of components lies in $\Lambda$. We write this as $\mul_{\Lambda}(A(x))$. In particular,
$$\mul_{0}(A(x))=1, \; \;  \; \mul_{1}(A(x))=A(x),\;\; \; \mul_{2}(A(x))=\frac{1}{2}\left(A(x)^2+A(x^2)\right).$$
The notation $\mul_{\geq k}$ refers to the multiset operator restricted to $\Lambda=\{k,k+1,...\}$.
\\
\paragraph{\textbf{The Dissymmetry Theorem for trees}} The Dissymmetry Theorem for trees (see~\cite{species}) provides a general methodology to relate a combinatorial class of unrooted trees with given properties to the corresponding classes of rooted trees.
More precisely, let $\mathcal{T}$ be a class of unrooted trees.
We define the following families of rooted trees: $\mathcal{T}_{\circ}$ is built from $\mathcal{T}$ by rooting a vertex,
$\mathcal{T}_{\circ-\circ}$ is the class of trees where an edge of $\mathcal{T}$ is rooted and $\mathcal{T}_{\circ\rightarrow\circ}$ is the class of trees obtained from $\mathcal{T}$ by rooting and orienting an edge.
The Dissymmetry Theorem for trees asserts that
\begin{equation}\label{eq:dys-theorem-abstract}
 \mathcal{T} \cup \mathcal{T}_{\circ\rightarrow\circ} \simeq
 \mathcal{T}_{\circ-\circ}\cup \mathcal{T}_{\circ},
\end{equation}
where ``$\simeq$'' means that there a bijection between the two combinatorial classes which translates directly into equalities of the corresponding generating functions.

\subsection{Asymptotic estimates and analytic combinatorics} \label{subsec: prelim-asympt}

By means of analytic methods we can obtain asymptotic estimates for $[x^n]A(x)$ in terms of the singularities of $A(x)$ with minimum complex modulus. Such singularities are called \Defn{dominant}.
Whenever $A(x)$ has non-negative coefficients, one of its dominant singularities (if there is any) is a positive real number by Pringsheim's Theorem, see~\cite[Thm.~IV.6]{FlajoletSedgewick}.

With this language, we obtain the asymptotic expansion of $[x^n]A(x)$ by \emph{transferring} the behaviour of $A(x)$ around its dominant singularity from a simpler function~$B(x)$ for which we know the asymptotic behaviour of the coefficients.
The first result in this direction is the \emph{Transfer Theorem for singularity analysis}~\cite{FlajoletOdlyzko,FlajoletSedgewick}. For our purposes we present a version of the theorem that covers the case when there is a unique dominant singularity $\rho$.

\begin{theorem}[Transfer Theorem for a unique dominant singularity~\cite{FlajoletOdlyzko}, simplified version]
\label{theo:transfer}
Assume that the generating function $A(x)$ is analytic in a dented domain $\Delta(\phi,R)$ at $\rho\in \mathbb{C}$, defined as the set
\[
\{x\in \mathbb{C}: x \neq \rho,\, |x|<R,\, |\mathrm{Arg}(x-\rho)|>\phi\},
\]
for $|\rho|<R\in \mathbb{R}$ and~$0<\phi<\pi/2$.
If $A(x)$ admits an expansion of the form
\[
A(x) = C\left(1 - \frac{x}{\rho} \right)^{-\alpha}+ O\left(\left(1-\frac{x}{\rho}\right)^{-\alpha+1}\right)
\]
for $x\rightarrow \rho$ in the dented domain~$\Delta(\phi, R)$ at $\rho$, and $\alpha \notin \{0,-1,-2,\dots\}$ then
\[
[x^n]A(x) = C\frac{1}{\Gamma(\alpha)} \cdot n^{\alpha-1} \cdot \rho^{-n} \, (1+o(1)),
\]
where~$\Gamma(s) = \int_0^\infty t^{s-1} e^{-t} dt$ denotes the classical Gamma function.
\end{theorem}

In the next sections we also have to analyze systems of functional equations. The main reference for this topic is the paper~\cite{Drmota}. For convenience, we rephrase it here in a simplified version (the interested reader could find the more general result in~\cite[Sec.~2.2.5.]{Drmota-book}).

Let $y_1(x),\dots, y_k(x)$ be generating functions satisfying a system of functional equations.
We define the vector $\y(x):=(y_1(x),\dots, y_k(x))$, and a system $\y=\textbf{F}(x; \y)$ satisfied by $\y(x)$. Notice that $\textbf{F}(x,\y)=(F_1(x,\y),\dots, F_k(x,\y))$.
We assume that each $y_i(x)$ is analytic at $x=0$, and that $y_i(0)=0$.
We also assume that all $F_i(x,\y)$ are analytic around 							$(0,\0)$ and have nonnegative Taylor coefficients around $(0,\0)$ (this condition assures the uniqueness of the solution).

The \emph{dependency graph} $G=(V,\mathcal{D})$ associated to the system $\y=\textbf{F}(x; \y)$ is the oriented graph whose vertex set is $V=\{y_1,\dots, y_k\}$ and the arc $\overrightarrow{y_{i}y_{j}}\in \mathcal{D}$ if and only if $\frac{\partial F_i(x,\y)}{\partial y_j}\neq 0$ (this indicates that $F_i(x,\y)$ really depends on $y_j$).
A dependency graph is called \emph{strongly connected} if every pair of vertices is connected by a directed path. With this terminology we have the following result:
\begin{theorem}[Singularity analysis of systems of functional equations~\cite{Drmota}, simplified version]\label{thm:drmota}
Let $\y(x)=\textbf{F}(x; \y(x))$ be a system of functional equations satisfying the conditions described above. Additionally, assume that the related dependency graph is strongly connected. Denote by $\textbf{I}_{k}$ the $k \times k$ identity matrix and by $\mathrm{Jac}(\textbf{F})$ the $k\times k$ Jacobian matrix associated to $\textbf{F}(x,\y)$. If the system
\begin{align}\label{eq:drmota}
\begin{cases}
\y&=\textbf{F}(x; \y) \\
 0&= \det\left(\textbf{I}_{k}-\mathrm{Jac}(\textbf{F})\right)
 \end{cases}
\end{align}
has a unique positive real solution $(x_0,\y_0)$ in the region of analyticity of each component of $\textbf{F}(x,\y)$, then there is a unique solution $\y(x)$ to the system of functional equations. Moreover, the functions $y_i(x)$ have nonnegative coefficients and a square-root expansion in a domain dented at $x_0$.
\end{theorem}

\section{Matroid decomposition}\label{sect: matroids-decom}
This section is devoted to the analysis of the structure of $2$-level matroids. Every $2$-connected matroid has a tree decomposition which relies on the $2$-sum and we refer to \cite[Sect.~8.3]{Oxley} for a complete overview on this topic. We state here the results which are relevant for the paper and we explore further features of tree decomposition that are specific for the class of $2$-level matroids. First let us make precise what we mean by a decomposition.
\begin{defi}
A \Defn{matroid-labelled tree} is a tree $T$ with vertex set $\{ \NN_1,\ldots , \NN_s\}$ for some positive integer $s$ such that
\begin{enumerate}[\rm (i)]
\item the $\NN_i$'s are matroids with pairwise disjoint ground sets;
\item an edge joining $\NN_i$ and $\NN_j$ is labelled by a set $\{e_i,e_j\}$ such that $e_i\in E(\NN_i)$, $e_j\in E(\NN_j)$, and $e_i$, $e_j$ are neither loops nor coloops;
\item the labels of the edges of $T$ are pairwise disjoint.
\end{enumerate}
We call $\NN_1,\ldots , \NN_s$ the \Defn{vertex labels} of $T$.
\end{defi}

\begin{example}\label{ex:matrlabtree}
Let us consider the matroid-labelled tree in the picture whose vertex labels are all graphic matroids.
\begin{figure}[htb]
\begin{center}
\includegraphics[width=7.5cm]{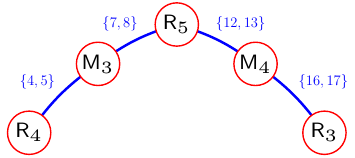}
\caption{Matroid-labelled tree with vertex labels of type ring and multiedge.}
\label{fig:matroid-labelled}
\end{center}
\end{figure}
In particular they are rings and multiedges. Each vertex label must be provided with its ground set and its collection of bases. For instance the ring $R_4$ has ground set $E(R_4)=\{ 1,2,3,4\}$ and collection of bases $\binom{[4]}{3}$. For a complete description of the vertex labels we refer to Example \ref{ex:graphdecomp}.
\end{example}

For a matroid-labelled tree $T$, we can contract an edge $t$ labelled by $\{e_i,e_j\}$ connecting two vertex labels $\NN_{i}$ and $\N_{j}$. The result is a matroid-labelled tree $T/t$ with the same edges and vertex labels, except that the vertex labels $\NN_{i}$ and $\NN_{j}$ have been gathered into a unique vertex label, namely ${(\NN_{i},e_{i})\oplus_2 (\NN_{j},e_{j})}$, and the edge $t$ has been contracted.

\begin{example}\label{ex:graphdecomp}
The vertex labels of the matroid-labelled tree introduced in Example \ref{ex:matrlabtree} are all graphic matroids. Thus, we can represent it as a sequence of $2$-sums of graphs. In the picture we specify the ground set for each of the graphs.
\begin{figure}[htb]
\begin{center}
\includegraphics[width=14cm]{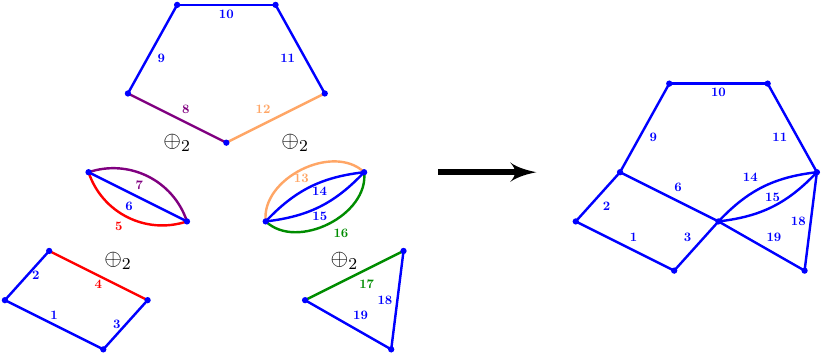}
\caption{Graph obtained after contraction of all edges of the matroid-labelled tree.}
\label{fig:tree-decomposition}
\end{center}
\end{figure}

Contracting an edge in the matroid-labelled tree corresponds to computing the $2$-sum of two graphs. If we contract all the edges we get the graph on the right which happens to be a series-parallel graph. Definition \ref{def:treedecomp} shows that the matroid-labelled tree we are considering is a tree decomposition for the matroid associated to this series-parallel graph.
\end{example}

\begin{defi}\label{def:treedecomp}
A \Defn{tree decomposition} of a $2$-connected matroid $\M$ is a matroid-labelled tree $T$ such that if $V(T)=\{\NN_1,\ldots, \NN_s \}$ and $E(T)=\{t_1,\ldots , t_{s-1}\}$, then
\begin{itemize}
\item $E(\M)=(E(\NN_1)\cup E(\NN_2) \cup \ldots \cup E(\NN_s))\setminus (t_1 \cup t_2 \cup \ldots \cup t_{s-1} )$;
\item $|E(\NN_i)|\geq 3 $ for all $i$, unless $|E(\M)|< 3$, in which case $s=1$ and $\NN_1=\M$;
\item $\M$ is the matroid that labels the single vertex of $T/\{t_1,t_2,\ldots, t_{s-1}\}$.
\end{itemize}
\end{defi}
We now report a theorem from \cite[Thm.~8.3.10]{Oxley} which first appeared in \cite{Cunningham}. According to our definitions, we replace the words ``circuit'' and ``cocircuit''  with ``ring'' and ``multiedge'', respectively.

\begin{theorem}\label{thm:uniquedecomp}
Let $\M$ be a $2$-connected matroid. Then $\M$ has a tree decomposition $T_\M$ in which every vertex label is $3$-connected, a ring, or a multiedge, and there are no two adjacent vertices that are both labelled by rings or are both labelled by multiedges. Moreover, $T_\M$ is unique up to relabelling of its edges.
\end{theorem}
In order to obtain the uniqueness, it is necessary to require that there are no two adjacent vertex labels that are both rings or multiedges, otherwise adjacent rings (or multiedges) could make possible to keep the same tree structure while changing the vertex labels. These additional requirements to get uniqueness justify why we consider separately the labels of type $\unif$, $\multiedge$, and $\ring$ in Section \ref{sect:counting}.

The theorem allows us to uniquely represent every matroid by a matroid-labelled tree whose vertex labels are $3$-connected matroids (except rings and multiedges). In this paper we want to tackle the problem from a different perspective: instead of starting with a matroid and finding its tree decomposition, the goal is to count how many non-isomorphic matroid-labelled trees can be constructed from a given set of possible vertex labels. In this constructive process, every time that we establish the adjacency of two vertices, we have to decide one element for each ground set of the two vertex labels to be the base points of the $2$-sum.

As shown in Example \ref{ex:P6}, the choice of the elements affects the result of the $2$-sum: there exist two non-isomorphic matroids whose tree decompositions have the same tree structure and the same vertex labels, but different labels for the edges of the tree.

Before presenting the example, let us give an explicit description of the collection of circuits of the matroid $(\M_1,e_1)\oplus_2 (\M_2,e_2)$, namely
{\small \begin{equation}\label{eqn:circuitsof2sum}
\mathcal{C}(\M_1\setminus e_1)\cup \mathcal{C}(\M_2 \setminus e_2)\cup \left\{(C_1-\{e_1\})\cup (C_2- \{e_2 \}) \; : \; e_1\in C_1 \in \mathcal{C}(\M_1) \textrm{ and }e_2\in C_2 \in \mathcal{C}(\M_2) \right\}.
\end{equation}}
\begin{example}\label{ex:P6}
Consider the $3$-connected matroid $P_6$ described in Subsection \ref{subsec:matroid}. Construct a matroid-labelled tree with two adjacent vertex labels, both equal to $P_6$ (see Figure \ref{fig:P6}). Let us label the ground set of the first copy of $P_6$ from $1$ to $6$ and the ground set of the second copy from $7$ to $12$. Each of the two copies has one circuit of length $3$ (we assume $\{1,2,3\}$ and $\{7,8,9\})$ and all the other circuits are of length $4$.
\begin{figure}[htb]
\begin{center}
\includegraphics[width=10.5cm]{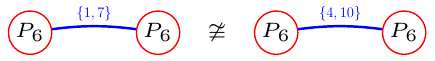}
\caption{Non-isomorphic matroids obtained as 2-sum of two matroids of type $P_6$.}
\label{fig:P6}
\end{center}
\end{figure}

Consider the matroid $(P_6,1)\oplus_2 (P_6,7)$. It has circuits of length $4$, $5$, $6$. On the other hand, the matroid $(P_6,4)\oplus_2 (P_6,10)$ has circuits of length $3$, $4$, $6$. Thus, the two matroids are not isomorphic.
\end{example}
Nevertheless if we focus on $2$-level matroids, the vertex labels are chosen among uniform matroids that we divide in the following three categories:
\begin{enumerate}[\rm(i)]
\item $\multiedge$-vertices: correspond to multiedges of size at least $3$;
\item $\ring$-vertices: correspond to rings of size at least $3$;
\item $\unif$-vertices: correspond to uniform matroids $U_{n,k}$ such that $n\geq 4$ and $2\leq k \leq n-2$.
\end{enumerate}
We define a new class of trees as follows.
\begin{defi}
Let $T$ be a tree whose vertex labels are of type $\unif$, $\multiedge$, and $\ring$ and such that no two $\multiedge$-vertices and no two $\ring$-vertices are adjacent. The tree $T$ is a \Defn{$\UMR$-tree} if $\deg(\NN_i)\leq |E(\NN_i)|$ for every vertex label $\NN_i$.
\end{defi}
For this particular class, the tree structure and the vertex labels are enough to determine the matroid uniquely up to matroid isomorphism. The proof of this fact is provided by Lemma \ref{lemma:transp_inv} and Lemma \ref{lemma:node_inv_preserved}. The main result of this section is the following theorem, which is required for the enumeration in Section \ref{sect:counting}.

\begin{theorem}\label{thm:main-matroid}
The family of $2$-connected $2$-level matroids is in bijection with the family of $\UMR$-trees.
\end{theorem}

Before presenting the proof of the theorem, we introduce some further definitions. For each vertex label $\NN_i$ of the tree decomposition of a matroid $\M$, we partition the ground set $E(\NN_i)=\{e_1,\ldots , e_{s_i}\}$ into two sets: the set $W(\NN_i)$ of elements which are base points for the $2$-sum with a vertex label adjacent to $\NN_i$ and the set $F(\NN_i)=E(\NN_i)\setminus W(\NN_i)$. We call $W(\NN_i)$ the set of \Defn{ideal elements} (generalizing the notion of ideal edge in \cite[Sect.~IV.3]{Tutte}) and $F(\NN_i)$ the set of \Defn{free elements}. Note that the ideal elements do not belong to the ground set of $\M$, while we have $E(\M)=\cup_{i} F(\NN_i)$.

For a matroid $\M$ let us consider the set of its circuits $\mathcal{C}$. We say that $\M$ is \Defn{transposition invariant} with respect to the pair of elements $\{e_1,e_2\}\subset E(\M)$ if we have that $\pi(\mathcal{C})=\mathcal{C}$, where $\pi$ is the transposition $(e_1 , e_2)$ and
$$ \pi(\mathcal{C})=\{\pi(C) \; : \; C\in \mathcal{C} \}. $$
The notation $\pi(C)$ means that we apply the permutation of the ground set $\pi : E(\M)\rightarrow E(\M)$ to the circuit $C$ of $\M$. A matroid is \Defn{permutation invariant} if it is transposition invariant with respect to every pair of elements in the ground set.
\begin{example}
Every uniform matroid $U_{n,k}$ is permutation invariant, since for every choice of $e_1,e_2\in [n]=E(U_{n,k})$, $\pi(\cdot)$ is a bijection from the set of $(k+1)$-subsets of $[n]$ to itself. Moreover if a matroid $\M=([n],\B)$ is permutation invariant, then it is a uniform matroid. Indeed let $C$ be the circuit with the least number $s$ of elements, then all the other subsets $\binom{[n]}{s}$ have to be circuits (by transposition invariance). It also follows that there cannot be other circuits. Thus $\M=U_{n,s-1}$.
\end{example}
We say that $\M$ is $\NN_i$-\Defn{transposition invariant}, for $\NN_i$ vertex label of the tree decomposition if it is transposition invariant with respect to every pair of elements in $F(\NN_i)$. We say that $\M$ is \Defn{node-invariant} if it is $\NN_i$-transposition invariant for every vertex label $\NN_i$ of the tree decomposition.
\begin{lemma}\label{lemma:transp_inv}
Let $\M$ be a $\NN_i$-transposition invariant $2$-connected matroid and $\U$ a uniform matroid. For any choice of $f\in F(\NN_i)$ and $u\in E(\U)$, the $2$-sum $(\M,f) \oplus_2 (\U,u)$ yields the same matroid up to isomorphism.
\end{lemma}
\begin{proof}
The uniform matroid $\U$ is permutation invariant and thus the choice
of $u \in E(\U)$ does not affect the result of the 2-sum. Consider any two elements $f_1,f_2\in F(\NN_i)$. We want to show that
\[
S_{f_1}:=(\M,f_1) \oplus_2 (\U,u)\cong (\M,f_2) \oplus_2 (\U,u)=:S_{f_2}.
\]
Notice that $E(S_{f_2})=E(S_{f_1})- \{ f_2 \} \cup \{ f_1 \}$. We claim that the bijection\\${\varphi \ : E(S_{f_1}) \rightarrow E(S_{f_2})}$ such that
\[
\varphi(e)=\begin{cases}
f_1 & \mbox{,  if }  e= f_2 \\
e  & \mbox{, otherwise}
\end{cases}
\]
yields the matroid isomorphism. We need to show that for every $X\subset E(S_{f_1})$, $X\in \mathcal{C}(S_{f_1})$ if and only if $\varphi(X)\in \mathcal{C}(S_{f_2})$.

As we have seen in ~\eqref{eqn:circuitsof2sum} a circuit $C$ of $S_{f_1}$ can be of $3$ different types:
\begin{itemize}
\item $C\in \mathcal{C}(\U \setminus u)$. In this case $\varphi(C)=C$ and clearly $C\in \mathcal{C}(S_{f_2})$.
\item $C\in \mathcal{C}(\M \setminus f_1)$. This implies that $C$ is a circuit of $\M$, $f_1\notin C$. Since $\M$ is $\NN_i$-transposition invariant, we have that $\pi(C)\in \mathcal{C}(\M)$ for $\pi=(f_1 , f_2)$. Moreover, $f_1\notin C$ implies $f_2\notin \pi(C)$, that is $\pi(C)\in \mathcal{C}(\M\setminus f_2)$. Finally, $\varphi(C)=\pi(C)\in \mathcal{C}(\M\setminus f_2)$ and thus $\varphi(C)\in \mathcal{C}(S_{f_2})$.
\item $C=(C_1-\{f_1\})\cup (C_2-\{ u \})$, $f_1\in C_1 \in \mathcal{C}(\M)$ and $u\in C_2 \in \mathcal{C}(\U)$. Since $\M$ is $\NN_i$-transposition invariant, for $\pi=(f_1 , f_2)$ we have $\pi(C_1) \in \mathcal{C}(\M)$ and $f_2\in \pi(C_1)$. Moreover, $\varphi(C)=(\pi(C_1)-\{ f_2\})\cup (C_2-\{ u \})$, $f_2\in \pi(C_1) \in \mathcal{C}(\M)$ and $u\in C_2 \in \mathcal{C}(\U)$ and thus $\varphi(C)\in \mathcal{C}(S_{f_2})$.
\end{itemize}
The same argument applies to check that all circuits of $S_{f_2}$ are circuits of $S_{f_1}$ under the map $\varphi^{-1}$. This concludes the proof.
\end{proof}
\begin{lemma}\label{lemma:node_inv_preserved}
Let $\M$ be a node-invariant $2$-connected matroid and $\U$ a uniform matroid. The $2$-sum $(\M,f) \oplus_2 (\U,u)$ is a node-invariant matroid for any choice of $f\in E(\M)$ and $u\in E(\U)$.
\end{lemma}
\begin{proof}
Choose a vertex label $\NN_i$ of the unique tree decomposition $T_\M$ of $\M$. Without loss of generality, let us assume $f\in F(\NN_i)$. To prove that $S_f:=(\M,f)\oplus_2 (\U,u)$ is node-invariant, we need to check the transposition invariance for each vertex label. For any vertex label $\NN_j$ and $f_1,f_2\in F(\NN_j)$, $f_1,f_2\neq f$, we have that the set $\mathcal{C}(S_f)$ is invariant under $\pi=(f_1 ,  f_2)$. Indeed $\C(\M\setminus f)$ and $\C(\U\setminus u)$ are invariant under $\pi$ because $\M$ is node-invariant and $\U$ is permutation invariant. The same holds true for the circuits of the third type, since
$$\pi((C_1- \{f\}) \cup (C_2 -\{ u\}))=(\pi(C_1)-\{ f\}) \cup (C_2 -\{ u\})$$
and $f\in \pi(C_1) \in \mathcal{C}(\M)$ by node-invariance of $\M$.

The tree decomposition of $S_f$ has one new node, labelled by the uniform matroid $\U$. We still have to check that $S_f$ is $\U$-transposition invariant. The same argument used above applies to $\U$, since it is a permutation invariant matroid.
\end{proof}

\begin{proof}[Proof of Theorem~\ref{thm:main-matroid}:]
Let us start with a uniform matroid $\NN_1$. Since $\NN_1$ is permutation invariant, it is also node-invariant. The $2$-sum of $\NN_1$ with a second uniform matroid $\NN_2$ yields a node-invariant matroid by Lemma \ref{lemma:node_inv_preserved}. We can iteratively add by $2$-sum new uniform matroids $\NN_3,\,\NN_4,\,\ldots,\, \NN_s$. The matroid we get at every step is clearly node-invariant. Moreover, at the $j$-th iteration we have to select which vertex label $\NN_i$, $i<j$ of $\M$ is adjacent to $\NN_j$. Once we fix $\NN_i$, the resulting matroid $(\M,f)\oplus_2(\NN_j,e)$ does not depend (up to isomorphism) on the choice of $f\in F(\NN_i)$ by Lemma \ref{lemma:transp_inv}. We can conclude that the structure of the tree decomposition and the vertex labels are enough to determine uniquely the $2$-level matroid. Vice versa Theorem \ref{thm:uniquedecomp} together with Theorem \ref{thm:unifdec2lev} proves that a $2$-connected $2$-level matroid uniquely identifies a tree structure with vertex labels chosen among the uniform matroids.
\end{proof}

We close the section with a proposition from \cite[Prop.~7.1.22]{Oxley} which is needed to deal with self-duality in \ref{subsec:selfdual}.
\begin{proposition}\label{prop:dual_matroid}
Let $\M_1$ and $\M_2$ be two matroids and $e_i\in E(\M_i)$. Then
\[
((\M_1,e_1)\oplus_2 (\M_2,e_2))^*=(\M_1^*,e_1) \oplus_2 (\M_2^*,e_2).
\]
\end{proposition}

\section{Counting $\UMR$-trees}\label{sect:counting}
In this section we apply the results in Section \ref{sect: matroids-decom} to get enumerative formulas for the number of $2$-level matroids of fixed size.
By means of Theorem \ref{thm:main-matroid}, this is equivalent to the enumeration of $\UMR$-trees. To the set of $\unif$-vertices, $\multiedge$-vertices, and $\ring$-vertices,
we add an additional type of vertices that we call \emph{legs}. Legs always have degree $1$, and are graphically represented by small red disks. For each free element of a vertex label $\NN_i$ we draw a leg connected to $\NN_i$. Observe that legs represent all the leaves of the tree. Hence, we develop enumerative formulas in terms of the number of legs in our tree mode and we translate them into counting results in the matroid setting.
The combinatorial restrictions we consider in our trees (which naturally arise from the obstructions inherited from the matroid setting) are the following:
\begin{enumerate}
\item The edges are unlabelled;
\item No two $\ring$-vertices and no two $\multiedge$-vertices are adjacent;
\item The degree of the $\ring$-vertices and $\multiedge$-vertices is greater or equal than $3$, and the degree of the $\unif$-vertices is greater or equal than $4$.
\end{enumerate}

In principle, our goal is to get enumerative formulas for $\UMR$-trees, but in order to apply the Dissymmetry Theorem for trees (Section \ref{sec:prelim}) we need to encode rooted families.
For this reason we introduce the following technical definition: a $\UMR$-tree is said to be \emph{pointed} if it has a special leaf of size $0$ (namely, it does not contribute to the total amount of legs) that we call \emph{virtual leg}.
Roughly speaking, the virtual leg pinpoints its adjacent vertex which we call the \emph{pointed vertex} of the $\UMR$-tree.
We use a red triangle to graphically represent the virtual leg.
See Figure~\ref{fig:tree1} for an example of a pointed $\UMR$-tree.
\begin{figure}[htb]
\begin{center}
\includegraphics[width=9.5cm]{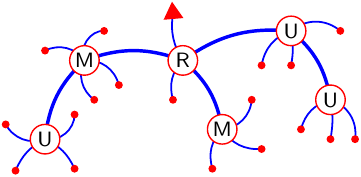}
\caption{A pointed tree with 18 legs, 1 virtual leg, and a pointed $\ring$-vertex.}
\label{fig:tree1}
\end{center}
\end{figure}

If a vertex is incident with the virtual leg, its \emph{restricted degree} is the total degree minus $1$.
Notice that $\unif$-vertices have \emph{multiplicity} due to the rank of the associated matroid. In other words, once the total degree of a $\unif$-vertex is fixed (call it $d$), then the possible rank could take any value in $\{2,3,4,\ldots, d-2\}$.
This yields $d-3$ possible different uniform matroids for this $\unif$-vertex.

In the next subsections we use ordinary generating functions to enumerate $\UMR$-trees. We first analyze the rooted case and then the unrooted case; in both cases the variable $x$ encodes non-virtual legs.

\subsection{Counting pointed $\UMR$-trees}\label{subsec:GF-counting-rooted}

We denote by $\AR(x),\, \AM(x)$ and $\AU(x)$ the generating functions for pointed trees where the virtual leg is adjacent to a $\ring$-vertex, a $\multiedge$-vertex, and a $\unif$-vertex, respectively.
Additionally, we write $A_l(x)$ for the generating function of the elementary tree pointed at a leg. Clearly, $A_l(x)=x$.
Observe that the first non-zero coefficients in the generating functions of pointed $\UMR$-trees are $[x^2]\AR(x)=[x^2]\AM(x)=1,$ and $[x^3]\AU(x)=1$.

We start getting relations between these generating functions by decomposing the trees at the pointed vertex.
Let us start with $\AR(x)$. Observe that such a tree can be described as a $\ring$-vertex (the pointed one) followed by a multiset of size greater or equal than $2$ of (the disjoint union of) trees rooted at either a $\multiedge$-vertex or at a $\unif$-vertex as shown in the following figure.

\begin{figure}[htb]
\begin{center}
\includegraphics[width=14.7 cm]{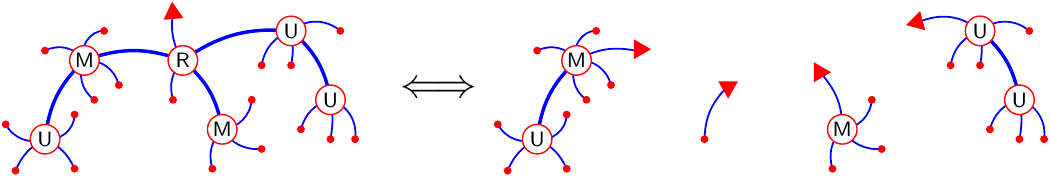}
\caption{Decomposition of a pointed $\UMR$-tree.}
\label{fig:decompofumrtree}
\end{center}
\end{figure}

The combinatorial description gives us that $\AR(x)= \mul_{\geq 2} (\AM(x)+\AU(x)+A_l(x)).$
This equation can be made explicit by means of the multiset operator:
\begin{align}\label{eq: AR}
\AR(x)= \exp\left(\sum_{r=1}^{\infty} \frac{1}{r} \left(\AM(x^r)+\AU(x^r)+A_l(x^r)\right)\right) - 1-\left(\AM(x)+\AU(x)+A_l(x)\right).
\end{align}
A similar argument holds changing the pointed $\ring$-vertex by a $\multiedge$-vertex. This gives an analogous equation for $\AM(x)$:
\begin{align}\label{eq: AM}
\AM(x)=\exp\left(\sum_{r=1}^{\infty} \frac{1}{r} \left(\AR(x^r)+\AU(x^r)+A_l(x^r)\right)\right) - 1-\left(\AR(x)+\AU(x)+A_l(x)\right).
\end{align}
Observe that Equations~\eqref{eq: AR} and~\eqref{eq: AM} give that $\AR(x)=\AM(x)$.
Indeed, by subtracting Equation \eqref{eq: AR} from Equation \eqref{eq: AM} we obtain that
$$\sum_{r\geq 1}\frac{1}{r}\AR(x^r)=\sum_{r\geq 1}\frac{1}{r}\AM(x^r).$$
These two formal power series have the same coefficients. In particular, for each choice of $n$, $[x^n]\sum_{r\geq 1}\frac{1}{r}\AR(x^r)= [x^n]\sum_{r\leq n}\frac{1}{r}\AR(x^r)$.
Now, applying an easy induction argument we can conclude that for each $n$, $[x^n]\AR(x)=[x^n]\AM(x)$.

Getting formulas for $\AU(x)$ is slightly more involved: if the pointed $\unif$-vertex has total degree $d$, then it has multiplicity $d-3$.
This fact must be encoded in the counting formulas.
Let us use an auxiliary variable $u$ which marks the restricted degree of the pointed $\unif$-vertex (namely, the total degree $d$ minus $1$).
Here we emphasize that we do \emph{not} consider the contribution of the virtual leg to the total number of legs $n$. This is due to technical reasons that are going to be clear while proceeding with the counting.
However, the multiplicity of the pointed $\unif$-vertex must be considered with respect to the total degree of the vertex (thus including the virtual leg) and not with respect to the restricted degree.
Indeed, for a pointed $\unif$-vertex of degree $d$, its restricted degree is equal to $r=d-1$, and its multiplicity is equal to $d-3=r-2$.

We write $a_{n,r}$ for the number of pointed trees with $n$ non-virtual legs whose virtual leg is adjacent to a $\unif$-vertex of restricted degree $r$. The notation $\aU(x,u):=\sum_{n,\,r\geq 3}a_{n,r}x^n u^r$ refers to the corresponding generating function. Then we have
\begin{equation}\label{eq: AU}
 \AU(x)= \left.\sum_{n,r\geq 3}(r-2) a_{n,r}x^n u^r\right|_{u=1}=\left.\frac{\partial}{\partial u}\aU(x,u)\right|_{u=1}- 2\aU(x,1)
\end{equation}
Observe now that $\aU(x,u)$ satisfies the equation $\aU(x,u)=\mul_{\geq 3} (u(\AM(x)+\AR(x)+\AU(x)+A_l(x)))$, which arises from the fact that the pointed $\unif$-vertex has restricted degree $\geq 3$ (or equivalently, degree $\geq 4$). Hence we have that
\begin{align*}
\aU(x,u)&=\exp\left(\sum_{r=1}^{\infty} \frac{u^r}{r} \left(\AR(x^r)+\AM(x^r)+\AU(x^r)+A_l(x^r)\right)\right)\\
&- 1-u\left(\AR(x)+\AM(x)+\AU(x)+A_l(x)\right)-\mul_{ 2} (u(\AM(x)+\AR(x)+\AU(x)+A_l(x))).
\end{align*}
Now by using Equation~\eqref{eq: AU} we can write $\AU(x)$ in terms of $\aU(x,1)$ and its derivative at $u=1$:
\begin{align}\label{eq: AU2}
\AU(x)&=\exp\left(\sum_{r=1}^{\infty} \frac{1}{r} \left(\AR(x^r)+\AM(x^r)+\AU(x^r)+x^r\right)\right)\left(\sum_{r=1}^{\infty}  \left(\AR(x^r)+\AM(x^r)+\AU(x^r)+x^r\right)\right)\\ \nonumber
&-\left(\AR(x)+\AM(x)+\AU(x)+x\right)-2\mul_{ 2}(\AM(x)+\AR(x)+\AU(x)+x)\\ \nonumber
&-2\exp\left(\sum_{r=1}^{\infty} \frac{1}{r} \left(\AR(x^r)+\AM(x^r)+\AU(x^r)+x^r\right)\right)+\\ \nonumber
&+2+2\left(\AR(x)+\AM(x)+\AU(x)+x\right)+2\mul_{ 2}(\AM(x)+\AR(x)+\AU(x)+x).
\end{align}

Hence, we have three equations relating $\AR(x)$, $\AM(x)$ and $\AU(x)$.
\\

\subsection{Application of the Dissymmetry Theorem}~\label{subsec:GF-unrooted}
We now proceed by applying the Dissymmetry Theorem for trees (see Subsection~\ref{sub:prelim_enumeration}) in order to express $\UMR$-trees in terms of rooted ones. Let $T(x)$ be the generating function of $\UMR$-trees, where $x$ marks legs. Write $T_{v}(x)$, $T_e(x)$, and $T_d(x)$ the generating functions associated to families of $\UMR$-trees with a rooted vertex, a rooted edge and a rooted and oriented edge, respectively. By the Dissymmetry Theorem for trees stated in Equation~\eqref{eq:dys-theorem-abstract}, we have that
\begin{equation}\label{eq:dys-tree}
T(x)=T_{v}(x)+T_e(x)-T_d(x).
\end{equation}
Let us compute each generating function in terms of the pointed families obtained in Subsection~\ref{subsec:GF-counting-rooted}. Let us start with $T_e(x)$. This can be written as:
$$T_e(x)=T_{\multiedge-\ring}(x)+T_{\multiedge-\unif}(x)+T_{\multiedge-\bullet}(x)+T_{\ring-\unif}(x)+T_{\ring-\bullet}(x)+T_{\unif-\unif}(x)+T_{\unif-\bullet}(x)$$
where the index of each term shows the type of the end vertices of the rooted edge (for instance, the first term $\ring{-}\multiedge$ means that the rooted edge has as end vertices a $\ring$-vertex and a $\multiedge$-vertex).
By cutting the rooted edge and pasting two virtual legs on the ends (see Figure~\ref{fig:tree2}), each term in the sum (with the exception of $T_{\unif-\unif}(x)$, which has an additional symmetry) is the product of the corresponding generating functions of pointed families.
\begin{figure}[htb]
\begin{center}
\includegraphics[width=14.5cm]{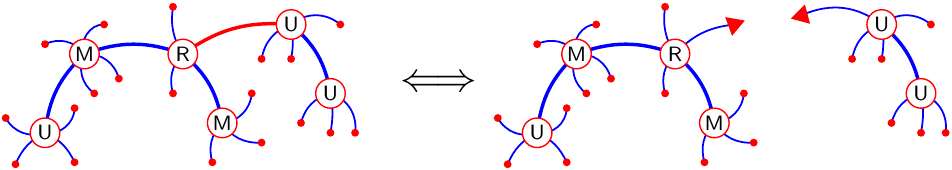}
\caption{A $\UMR$-tree rooted at an edge (colored red), and its decomposition in terms of pointed trees.}
\label{fig:tree2}
\end{center}
\end{figure}

The single situation where symmetry exists is in $T_{\unif-\unif}(x)$, and in this case we have a multiset of size $2$ of trees pointed at a $\unif$-vertex. We conclude that
\begin{equation}\label{eq:Te}
T_e(x)=\AM(x)(\AR(x)+\AU(x)+A_l(x))+\AR(x)(\AU(x)+A_l(x))+\mul_2(\AU(x))+A_l(x)\AU(x).
\end{equation}
A decomposition similar to the one of Equation~\eqref{eq:Te} applies for $T_d(x)$. Indeed this generating function can be written as:
\begin{align*}
T_d(x)&= T_{\multiedge \rightarrow \ring}(x)+T_{\multiedge\rightarrow \unif}(x)+T_{\multiedge\rightarrow \bullet}(x)\\
      &+T_{\ring\rightarrow \multiedge}(x)+T_{\ring\rightarrow \unif}(x)+T_{\ring \rightarrow \bullet}(x)\\
      &+T_{\unif\rightarrow \multiedge}(x)+T_{\unif\rightarrow \ring}(x)+T_{\unif\rightarrow \unif}(x)+T_{\unif\rightarrow \bullet}(x)\\
      &+T_{\bullet\rightarrow \multiedge}(x)+T_{\bullet\rightarrow \ring}(x)+T_{\bullet\rightarrow \unif}(x)
\end{align*}
where the index of each term shows the type of the end vertices for the rooted directed edge. In this situation the computations are similar and even easier, because there is no extra symmetry when dealing with an edge linking two $\unif$-vertices:
\begin{align}\label{eq:Td}
T_d(x)&=A_{M}(x)(A_{R}(x)+A_{U}(x)+A_l(x))\\ \nonumber
      &+ A_{R}(x)(A_{M}(x)+A_{U}(x)+A_l(x))\\ \nonumber
      &+A_{U}(x)(A_{M}(x)+A_{R}(x)+A_{U}(x)+A_l(x))\\ \nonumber
      &+A_l(x)(A_{R}(x)+A_{M}(x)+A_{U}(x)). \nonumber
\end{align}
The last generating function we want to get is $T_v(x).$
Observe that $T_v(x)$ is not the sum of the generating functions obtained in Subsection~\ref{subsec:GF-counting-rooted}, because now we do not have to consider the virtual leg.
We write
\begin{equation}\label{eq:Tv}
T_v(x)=T_\ring(x)+T_\multiedge(x)+T_\unif(x)+T_{\bullet}(x)
\end{equation}
where the index of each term indicates the type of the rooted vertex.
We want to express now each term by means of the previous pointed families.
It is obvious that
\begin{equation}\label{eq:Tbullet}
T_{\bullet}(x)=A_l(x)(A_\ring(x)+A_\multiedge(x)+A_\unif(x))
\end{equation}
because a rooted leg induces canonically a rooted edge.
Let us consider the other situations: observe that $T_\ring(x)=\mul_{\geq 3} (\AM(x)+\AU(x)+A_l(x))$, which is obtained by cutting the edges incident with the rooted $\ring$-vertex, and pasting a virtual leg to each resulting subtree. In particular
\begin{align}\label{eq:TR}
T_\ring(x)&=\mul_{\geq 3} (\AM(x)+\AU(x)+A_l(x))\\\nonumber
      &=\mul_{\geq 2} (\AM(x)+\AU(x)+A_l(x))-\mul_{2} (\AM(x)+\AU(x)+A_l(x))\\\nonumber
      &=\AR(x)-\mul_{2} (\AM(x)+\AU(x)+A_l(x))\nonumber
\end{align}
and, mutatis mutandis, an analogous expression holds for $T_\multiedge(x)$.
At last, let us study $T_\unif(x)$: let $\tU(x,u)$ be the generating function of trees with a rooted $\unif$-vertex, where the multiplicity of the rooted vertex is not encoded yet and $u$ encodes the degree of the rooted $\unif$-vertex. Then, $\tU(x,u)= \mul_{\geq 4}(u(\AR(x)+\AM(x)+\AU(x)+A_l(x)))$ and
\begin{align*}\tU(x,u)&=\sum_{n,d\geq 4}t_{n,d}x^n u^d \quad \Longrightarrow \quad \TU(x)= \left.\sum_{n,d\geq 4}(d-3) t_{n,d}x^n u^d\right|_{u=1}\\
&=\left.\frac{\partial}{\partial u}\tU(x,u)\right|_{u=1}- 3\tU(x,1).
\end{align*}
Applying the same trick we used for $\aU(x,u)$ in Subsection~\ref{subsec:GF-counting-rooted}, we get that
\begin{align}\label{eq:TU}
\TU(x)&=\left. \frac{\partial}{\partial u}\tU(x,u)\right|_{u=1}- 3\tU(x,1)\\ \nonumber
      &= \left. \left(\frac{\partial}{\partial u}-3\right) (\aU(x,u)-u^3 \mul_{3}(\AR(x)+\AM(x)+\AU(x)+A_l(x)))\right|_{u=1}\\ \nonumber
      &= \AU(x)-\aU(x,1)+(3-3)\mul_{3}(\AR(x)+\AM(x)+\AU(x)+A_l(x)))\\ \nonumber
      &=\AU(x)-\mul_{\geq 3}(\AR(x)+\AM(x)+\AU(x)+A_l(x))).\nonumber
\end{align}

Substituting Equations \eqref{eq:Tbullet},\eqref{eq:TR} and \eqref{eq:TU} in \eqref{eq:Tv} we get the expression for $T_{v}(x)$.
Finally, we replace \eqref{eq:Te}, \eqref{eq:Td} and this expression of $T_v(x)$ in Equation \eqref{eq:dys-tree}.
All together this brings us the generating function $T(x)$, whose first coefficients are $2x^3+4x^4+10x^5+27x^6+78x^7+246x^8+818x^9+2871x^{10}+10446x^{11}+39358x^{12}+\dots$

\subsection{Asymptotic analysis}\label{subsec:asympt}
Now we can apply the machinery arising from analytic combinatorics in order to get asymptotic estimates for $[x^n]T(x)$. The main point is based on studying the system of equations which defines $\AR(x), \AM(x)$ and $\AU(x)$, which provides the position and the nature of the dominant singularity of $T(x)$.

In particular, by means of the Drmota-Lalley-Woods methodology (see Subsection \ref{subsec: prelim-asympt}), we obtain the constant growth, which is $\rho^{-1}\approx 4.88052854$ (whose inverse $\rho\approx 0.20489584$ gives the radius of convergence around the origin of the generating function).
Possibly more important, we can show that all these generating functions have the same square-root singularity (see the details in the proof).
Moreover, the generating function $T(x)$ is an analytic expression of the previous counting formulas, hence the position of the singularity does not change. However, the type of the singularity changes due to a combinatorial cancellation arising from the Dissymmetry Theorem for trees applied to $\UMR$-trees.
Finally, the asymptotic estimates for the coefficients of $T(x)$ are deduced by means of the Transfer Theorem for singularity analysis (see Theorem \ref{theo:transfer}).

Before presenting the proofs, it is worth comparing the growth constant we get with the one arising in the context of unlabelled $2$-connected series-parallel graphs, which is the analogue in the graphical setting. In~\cite{DrmotaFusyKangKrausRue} it is proven that the number of unlabelled $2$-connected series-parallel graphs with $n$ vertices grows exponentially as $\gamma^{-n}$, where $\gamma\approx 0.12419991$ (and $\gamma^{-1}\approx 8.05153567$). Despite several similarities, there are few caveats that we have to keep into account:
\begin{enumerate}[\rm(i)]
\item matroids do not have a vertex structure; instead we count them by the number of elements in the ground set, which will also pay off when relating our results to the enumeration of $2$-level base polytopes;
\item the tree decompositions of series-parallel graphs have only $\ring$-vertices and $\multiedge$-vertices. General $2$-level matroids are constructed using also the $\unif$-vertices, that is, a much wider variety of building blocks.
\item  series-parallel graphs with different graph realizations can correspond to isomorphic matroids and must be counted only once in the matroid setting.
\end{enumerate}


The first result deals with the singular behaviour of $\AR(x), \AM(x)$ and $\AU(x)$:

\begin{proposition}\label{prop:main}
The generating functions $\AR(x), \AM(x)$ and $\AU(x)$ have a dominant singularity at $\rho\approx 0.20489584$. Additionally, this is the unique singularity in the region $\{x\in\mathbb{C}: |x|\leq \rho\}$. In a domain dented at $\rho$, $\AR(x)$, $\AM(x)$ and $\AU(x)$ have a singular expansion of the form
\begin{align*}
\AR(X)=\AM(X) =&\,\, A_0+A_1X+A_2X^2+A_3X^3+O(X^4), \\
       \AU(X) =& \,\, U_0+U_1X+U_2X^2+U_3X^3+O(X^4),
\end{align*}
 where $X=\sqrt{1-x/\rho}$, $A_0\approx 0.13529174$, $A_1\approx -0.23137622$, $A_2\approx 0.04653888$, $A_3\approx 0.06281332$, $U_0\approx 0.06921673$, $U_1\approx -0.19340420$, $U_2\approx 0.15045323$ and $U_3\approx 0.01018058$.

\end{proposition}

\begin{proof}
As we know that $\AR(x)=\AM(x)$, we just need to analyze the pair of equations \eqref{eq: AR} and \eqref{eq: AU2}.
Indeed if $\AR(x)$ and $\AU(x)$ have a unique singularity $\rho$, then the term
\begin{equation*}
\exp\left(\sum_{r=2}^{\infty} \frac{1}{r} \left(\AM(x^r)+\AU(x^r)+A_l(x^r)\right)\right).
\end{equation*}
in equation~\eqref{eq: AR} is analytic at $x=\rho$ (similarly in Equation~\eqref{eq: AU2}). Hence, we can approximate this term by its Taylor series (which can be computed by an iterative algorithm). As a result, we obtain a pair of functional equations in $x$, $\AR(x)$, and $\AU(x)$ satisfying the conditions of Theorem~\ref{thm:drmota}. Solving now the resulting system of 3 equations by means of \texttt{Maple} computations (namely, the two equations and the one associated to the jacobian matrix in Equation \eqref{eq:drmota}), we obtain the solution $x_0\approx 0.20489584$, $\solAR\approx 0.13529174$ and $\solAU\approx 0.06921673$. By Theorem~\ref{thm:drmota} the position of the singularity of both $\AR(x)$ and $\AU(x)$ is located at $\rho=x_0\approx 0.20489584$, and both $\AR(x)$ and $\AU(x)$ have a square-root expansion in a domain dented at $\rho$ of the form
\begin{align*}
\AR(X)=\AM(X) =&\,\, A_0+A_1X+A_2X^2+A_3X^3+O(X^4), \\
       \AU(X) =& \,\, U_0+U_1X+U_2X^2+U_3X^3+O(X^4)
\end{align*}
where $A_i,\,U_i$, $i\in\{1,2,3,4\}$ are computable constants.
In order to get approximate values of these constants, we substitute the square-root expansions of $\AM(x)=\AR(x)$ and $\AU(x)$ in equations \eqref{eq: AR} and \eqref{eq: AU2}. The terms of the form $\AM(x^r)=\AR(x^r)$ and $\AU(x^r)$ ($r\geq 2$) are also approximated by a truncation of the Taylor series (which can also be computed by an iterative algorithm), because these GFs are analytic at the point $x=\rho$. At this point we can get a system of equations in the $A_i$'s and the $U_i$'s by equating the coefficients with same degree of the square-root expansions. Solving this system yields the constants reported in the statement of the theorem.
\end{proof}

More precisely, we get the following result for $[x^n]T(x)$:
\begin{theorem}\label{thm:enum-trees}
The following asymptotic estimate holds:
$$[x^n]T(x)=C\cdot n^{-5/2}\cdot \rho^{-n}\,\,(1+o(1))$$
where $C\approx 0.07583455$ and $\rho\approx 0.20489584$ are computable constants.
\end{theorem}

\begin{proof}
We use the singular square-root expansions for $\AR(x)$, $\AM(x)$ and $\AU(x)$ obtained in Lemma~\ref{prop:main}, together with the expressions in equations \eqref{eq:dys-tree}-\eqref{eq:TU} in order to get the singular expansion of $T(x)$:
$$T(x)=T_0+T_2 X^2+T_3 X^3 +O(X^4),$$
with $T_0\approx 0.03457946$, $T_2\approx -0.18596384$ and $T_3\approx 0.17921766$.
Observe that the constant multiplying $X$ in this singular expansion is equal to $0$ (due to the unrooting process in the Dissymmetry Theorem for trees). Finally we apply the Transfer Theorem for singularity analysis over this singular expansion.
\end{proof}

\subsection{Dealing with duality. Proof of Theorem \ref{thm:main}}\label{subsec:selfdual}
The last part is devoted to show that the contribution of self-dual 2-level matroids is exponentially small compared to the estimates we obtained in the previous subsection.
Let $\M$ be a matroid with tree decomposition $T_\M$, then Proposition~\ref{prop:dual_matroid} implies that the tree decomposition of $\M^*$ has the same tree structure of $T_\M$. Moreover, we replace each vertex label $\NN_i$ with its dual matroid $\NN_i^*$.

We are interested in self-dual $2$-connected $2$-level matroids. The vertex labels are chosen among uniform matroids, and the operation of duality turns labels of type $\multiedge_n$ into labels of type $\ring_n$ and vice versa, and $U_{n,k}$-labels into $U_{n,n{-}k}$-labels. It is clear that the self-dual labels are of the form $U_{2n,n}$. Moreover, for technical reasons, we consider also virtual legs and legs to be self-dual.

Our goal is to estimate the contribution of the family of self-dual $\UMR$-trees (namely $\UMR$-trees associated to self-dual matroids) to the total number of $\UMR$-trees. To do that we start analyzing the pointed situation: we write $\AR(x)=\SR(x)+\NR(x)$, $\AM(x)=\SM(x)+\NM(x)$ and $\AU(x)=\SU(x)+\NU(x)$, where the generating functions $\SR(x),\,\SM(x)$ and $\SU(x)$ encode self-dual trees whose pointed vertex is a $\ring$-vertex, a $\multiedge$-vertex and a $\unif$-vertex, respectively.
The generating functions $\NR(x),\,\NM(x)$ and $\NU(x)$ are the ones encoding trees which are not self-dual.
Observe that in particular $\SR(x)=\SM(x)=0$, because the dual of each $\ring$-vertex is an $\multiedge$-vertex, and consequently there are no self-dual trees pointed at either a $\ring$-vertex or a $\multiedge$-vertex.

We also use a similar notation for unrooted trees.
We write $T(x)=S(x)+N(x)$, where $S(x)$ is the generating function associated to self-dual (unrooted) trees.

The next lemma tells us that the contribution of self-dual rooted trees is exponentially small.

\begin{lemma} \label{lem:self-rooted-trees}
The following estimate holds:

$$[x^n]\SU(x)= o([x^n]\AU(x)).$$

\end{lemma}

\begin{proof} We get and analyze equations for $\SU(x)$. In this situation, the pointed vertex is a $\unif$-vertex associated to a uniform matroid of the form $U_{2n,n}$. Hence, we notice that the degree of the pointed vertex determines the rank and, in particular, the multiplicity in the counting is 1. Moreover, the possible restricted degree of the vertex are clearly in the set $\Lambda=\{3,5,7,\dots\}$.

Now observe that the collection of pending pointed subtrees is a multiset of pairs of pointed trees such that one is the dual of the second, followed by a multiset of odd size of self-dual pointed trees. Hence,
\begin{align}\label{eq:SU1}
\SU(x)&= \mul_{\{3,5,7,\dots\}}(\SU(x)+A_l(x))\\
&+\mul_{\geq 1}(\AR(x^2)+\AM(x^2)+(\AU(x^2)-\SU(x^2))) \mul_{\{1,3,5,7,\dots\}}(\SU(x)+A_l(x)). \nonumber
\end{align}
Let $\eta$ be the radius of convergence of $S_U(x)$. It is obvious that $\eta \geq \rho$, because the family of self-dual pointed trees is counted in the family of $\unif$-pointed trees. We need to show that $\eta > \rho$.

Equation~\eqref{eq:SU1} can be analyzed in a similar way to the one we find in the proof of Proposition \ref{prop:main}.
However, for our purposes it is enough to bound the coefficients of $\SU(x)$ by means of crude estimates.
Observe that $\mul_{\geq 3}(\SU(x)+A_l(x)) \geq \mul_{\{3,5,7,\dots\}}(\SU(x)+A_l(x))$ and

\begin{align*}
&\mul_{\geq 1}(\AR(x^2)+\AM(x^2)+(\AU(x^2)-\SU(x^2))) \mul_{\geq 1}(\SU(x)+A_l(x))   \geq \\
&\mul_{\geq 1}(\AR(x^2)+\AM(x^2)+(\AU(x^2)-\SU(x^2))) \mul_{\{1,3,5,7,\dots\}}(\SU(x)+A_l(x)).
\end{align*}
Hence, if $s(x)$ satisfies the equation
\begin{align} \label{eq:s}
s(x)= \mul_{\geq 3}(s(x)+A_l(x))+\mul_{\geq 1}(\AR(x^2)+\AM(x^2)+\NU(x^2)) \mul_{\geq 1}(s(x)+A_l(x)),
\end{align}
then $\SU(x) \leq s(x)$. Observe also that by the combinatorial specification of $\UMR$-trees $s(x) \leq \AU(x)$.
Let $\gamma$ be the dominant real singularity of $s(x)$.
Observe that this singularity arises either from the square-root singularity of the terms $\AR(x^2)$, $\AM(x^2)$, $\NU(x^2)$ at $x$ equals to $\sqrt{\rho} \approx 0.45265421$ or from a branch point (smaller than $\sqrt{\rho}$) of Equation~\eqref{eq:s}.

In the second case, Equation \eqref{eq:s} can be written in the form $s(x)=F(x,s(x))$ after replacing all analytic terms by their truncated Taylor series. Any hypothetic branch point arises as a coalescence of the solutions $x\leq \sqrt{\rho}$ of the pair of equations $s=F(x,s)$, $1=F_{s}(x,s)$. Since there is no such solution (these computations have been done with \texttt{Maple} by taking $30$ coefficients in the Taylor series of $\AR(x^2)+\AM(x^2)+\NU(x^2)$), there is no branch point $\gamma$ strictly smaller than $\sqrt{\rho}$, and consequently the singularity of $s(x)$ arises from the singularity of the term $\AR(x^2)+\AM(x^2)+\NU(x^2)$.
To conclude, $[x^n]s(x)$ has exponential growth of order $\rho^{-n/2}$, which is exponentially small compared with $\rho^{-n}$.
\end{proof}

Once we know that the number of self-dual pointed trees is exponentially small compared to the total number of pointed trees, we can prove that the number of self-dual $\UMR$-trees is also exponentially small compared to the total number of $\UMR$-trees.

\begin{proposition}\label{eq:self-dual}
The following estimate hold:
$$[x^n]S(x)=o([x^n]T(x)).$$
\end{proposition}

\begin{proof}
To prove the statement we obtain a generating function $D(x)$ such that $S(x) \leq D(x)$ and that $[x^n]D(x)=o([x^n]T(x))$.
We split the class of self-dual trees by looking at the type of the center for each self-dual tree. The \Defn{center} of a connected graph is the set of vertices that minimize the maximal path-distance from other vertices in the graph. The center of a tree consists of a single vertex or two adjacent vertices (we say it is an edge).

Let us write $S(x)=S_{\circ}(x)+S_{\circ- \circ}(x)$, where $S_{\circ}(x)$ and $S_{\circ-\circ}(x)$ are the generating functions associated to self-dual trees whose center is a vertex and an edge, respectively. We analyze each case separately.

We start with self-dual trees whose center is a vertex.
In this situation the center is necessarily a $\unif$-vertex labelled by a matroid of type $U_{2n,n}$. In this case the degree of the pointed vertex determines the rank of the $\unif$-vertex, which has to be counted with multiplicity one.
Hence, we have the crude bound $S_{\circ}(x)\leq  \mul(\AR(x^2)+\AM(x^2)+\NU(x^2)) \mul(\SU(x)+A_l(x))$, whose radius of convergence by Lemma~\ref{lem:self-rooted-trees} is strictly bigger than $\rho.$

Let us study now self-dual trees whose center is an edge.
Consider the pair of pointed trees that arise when cutting the edge which plays the role of the center of the tree (and pasting a virtual leg).
Two situations may happen:
\begin{enumerate}
\item [(1)] Each tree is self-dual.
\item [(2)] Each tree is non self-dual, but one is the dual of the other.
\end{enumerate}

In both cases $(1)$ and $(2)$ we can easily find a bound and the sum of the upper bounds yields the function $D(x)$. Namely, $S_U(x)^2$ and $\AR(x^2)+\AM(x^2)+\NU(x^2)$, respectively.
Therefore $S_{\circ-\circ}(x)\leq \SU(x)^2+\AR(x^2)+\AM(x^2)+\NU(x^2)$.
Finally, again by Lemma~\ref{lem:self-rooted-trees}, the radius of convergence of $S_{\circ-\circ}(x)$ is strictly bigger than $\rho$.
Hence the result follows.
\end{proof}

We can now prove that there are exponentially many $2$-level polytopes coming from matroid base polytopes:
\begin{proof}[Proof of Theorem~\ref{thm:main}: ]
Every $2$-connected $2$-level matroid $\M$ on $n$ elements is, by definition, associated with a $2$-level base polytope $P_{\M}$. The $2$-connectedness implies that the dimension of the base polytope is $n{-}1$. By Theorem~\ref{Ardila} there is only another matroid with congruent base polytope, namely $\M^*$.

Denote by $L_2(n)$ the number of $2$-connected $2$-level matroids and by $S_2(n)$ the number of self-dual ones. The number of non-congruent $(n{-}1)$-dimensional $2$-level polytopes associated with such family is $\frac{L_2(n)+S_2(n)}{2}$. This yields a lower bound to the number of $(n{-}1)$-dimensional $2$-level polytopes.

Applying the structural result of Section~\ref{sect: matroids-decom} and using the notation of Subsection~\ref{subsec:GF-unrooted} we easily see that $L_2(n)=[x^n]T(x)$ and $S_2(n)=[x^n]S(x)$. We do not have closed formulas for the coefficients of the generating functions, but nevertheless we are able to provide asymptotic estimates: by Theorem~\ref{thm:enum-trees} the number of $\UMR$-trees is asymptotically equal to $C\cdot n^{-5/2}\cdot \rho^{-n}\,\,(1+o(1)),$ where $C\approx 0.07583455$ and $\rho\approx 0.20489584$ are computable constants.
Due to Proposition~\ref{eq:self-dual}, the contribution of self-dual $\UMR$-trees to this asymptotic is exponentially small.
Hence, the number of non self-dual $\UMR$-trees is asymptotically equal to the whole number of $\UMR$-trees.
Finally, the number of $\UMR$-trees up to the duality relation is half of this value plus the number of self-dual $\UMR$-trees.
So, Theorem \ref{thm:main} holds by dividing the previous bound by 2.
\end{proof}

To conclude, observe that we can use the singular expansion of $T(x)$ in order to get asymptotic estimates for the number of $2$-level matroids, including the non-connected ones.
This family corresponds with the multiset construction applied over $\UMR$-trees (namely, forests).
Hence, the generating function here is $\mul(T(x))=\exp\left(\sum_{r=1}^{\infty}\frac{1}{r}(T(x^r))\right)$.
Observe that $$\exp\left(\sum_{r=1}^{\infty}\frac{1}{r}(T(x^r))\right)=\exp(T(x))\exp\left(\sum_{r=2}^{\infty}\frac{1}{r}(T(x^r))\right),$$
and the second term is analytic at $x=\rho$.
Hence, in a domain dented at $x=\rho$ the singular expansion of $\mul(T(x))$ is equal to:
$$\mul(T(x))=\exp(T_0+T_2 X^2+T_3 X^3+ O(X^4)) \exp\left(\sum_{r=2}^{\infty}\frac{1}{r}(T(\rho^r))\right),$$
(see the singular expansion of $T(x)$ in the proof of Theorem \ref{thm:enum-trees}) which has the expression
$$\mul(T(x))=F_0+F_2 X^2+F_3X^3+O(X^4),$$
with $F_0 \approx 1.03526853$, $F_2 \approx -0.19252251$, $F_3 \approx 0.18553841$. Applying now Theorem \ref{theo:transfer} we conclude that
$$[x^n]\mul(T(x)) = C' \cdot n^{-5/2} \cdot \rho^{-n} (1+o(1)),$$
with $C' \approx 0.07850913$. Observe that the constant $C'$ is slightly bigger than the constant obtained in the asymptotic estimate for $\UMR$-trees.
\\
\paragraph{\textbf{Acknowledgments:}} the authors thank Raman Sanyal for inspiring discussions and for accurate reading of this paper. Francisco Santos and G\"unter Ziegler are also thanked for helpful comments and suggestions.

\end{document}